\documentclass[hidelinks,onefignum,onetabnum]{siamart220329}

\usepackage{subcaption}
\usepackage{color}
\usepackage{bm}
\usepackage{mathtools} 
\usepackage{lipsum}
\usepackage{amsfonts}
\usepackage{graphicx}
\usepackage{epstopdf}
\usepackage{algorithmic}
\usepackage{bbold}
\usepackage{amsmath}
\usepackage{amsopn}
\usepackage{amssymb}
\usepackage{multirow}
\usepackage{url,hyperref}

\newcommand{\hI}{\mathcal{I}}

\newcommand{\hN}{\mathcal{N}}
\newcommand{\Lk}{\mathcal{L}_{\kappa}}
\newcommand{\Linf}{L^{\infty}}
\newcommand{\e}{\mathrm{e}}
\newcommand{\rd}{\mathrm{~d}}

\newcommand{\bx}{{\bm{x}}}

\allowdisplaybreaks
\ifpdf
  \DeclareGraphicsExtensions{.eps,.pdf,.png,.jpg}
\else
  \DeclareGraphicsExtensions{.eps}
\fi

\newsiamremark{remark}{Remark}
\newsiamremark{hypothesis}{Hypothesis}
\crefname{hypothesis}{Hypothesis}{Hypotheses}
\newsiamthm{claim}{Claim}

\headers{MBP exponential integrators}{C. Quan, P. Zheng, and Z. Zhou}

\title{
Maximum bound preservation of exponential integrators for Allen--Cahn equations
}

\author{Chaoyu Quan\thanks{School of Science and Engineering, The Chinese University of Hong Kong, Shenzhen, 518172, Guangdong, People's Republic of China (\email{quanchaoyu@cuhk.edu.cn}).}  
  \and Pinzhong Zheng\thanks{Department of Applied Mathematics, The Hong Kong Polytechnic University, Kowloon, Hong Kong SAR 
 (\email{pinzhong.zheng@connect.polyu.hk}, \email{zhizhou@polyu.edu.hk}).}  
  \and Zhi Zhou\footnotemark[2]}

\ifpdf
\hypersetup{
  pdftitle={Maximum bound preservation of exponential integrators for Allen--Cahn equations},
  pdfauthor={C. Quan, P. Zheng, and Z. Zhou}
}
\fi

\begin{document}
\maketitle
\begin{abstract}
We develop and analyze a class of arbitrarily high-order, maximum bound preserving time-stepping schemes for solving Allen–Cahn equations. These schemes are constructed within the iterative framework of exponential integrators, combined with carefully chosen numerical quadrature rules, including the Gauss--Legendre quadrature rule and the left Gauss--Radau quadrature rule.  Notably, the proposed schemes are rigorously proven to unconditionally preserve the maximum bound without requiring any additional postprocessing techniques, while simultaneously achieving arbitrarily high-order temporal accuracy. A thorough error analysis in the $L^2$ norm is provided. Numerical experiments validate the theoretical results, demonstrate the effectiveness of the proposed methods, and highlight that an inappropriate choice of quadrature rules may violate the maximum bound principle, leading to incorrect dynamics.
\end{abstract}



\section{Introduction}\label{sec:intro}
In this work, we explore a class of maximum bound preserving (MBP) exponential integrators for solving the Allen--Cahn equation
\begin{equation}\label{eq:AC}
    \left\{
    \begin{aligned}
         & u_{t}      = \varepsilon^2 \Delta u+ f(u), &  & \bm{x} \in \Omega, ~  t \in(0,T], \\
         & u(\bm{x},0)  =u^0(\bm{x}),                    &  & \bm{x} \in \Omega ,
    \end{aligned}
    \right.
\end{equation}
subject to periodic boundary conditions. 
Here, $\Omega$ denotes a bounded domain in $\mathbb{R}^d~ (d=1,2,3)$, with $\Delta$ representing the Laplacian operator in $d$ dimensions. 
The unknown function $u$ denotes the phase variable, and the parameter $\varepsilon >0$ represents the interfacial width. 
The nonlinear term is given by $f(u) = -F^{\prime}(u)$, where $F$ is a double-well potential function. A typical example is the polynomial potential, given by $f(u) = u - u^3$.
In this paper, we consider a more general nonlinear function $f: \mathbb{R} \to \mathbb{R}$, assumed to be continuously differentiable and satisfy the condition
\begin{equation}\label{assump_f}
\exists \text{ a constant } \beta > 0 \text{ such that } f(\beta) \leq 0 \leq f(-\beta).
\end{equation}
It is well known that the Allen--Cahn equation has the \textit{maximum bound principle}, which states that if the initial and boundary data are bounded by $\beta$ in absolute value, then the absolute value of the solution remains bounded by $\beta$ everywhere and for all time. Additionally, the Allen--Cahn equation satisfies the \textit{energy dissipation law}. The Allen--Cahn equation \eqref{eq:AC} can be viewed as an $L^2$ gradient flow with respect to the energy functional:
\begin{equation*}
E(u)=\int_{\Omega}\left(\frac{\varepsilon^2}{2}|\nabla u|^2+F(u)\right) \rd  \bm{x},
\end{equation*}
and thus it holds that
\[
    \frac{\mathrm{d}}{\mathrm{d} t} E(u)=\left(\frac{\delta E(u)}{\delta u}, \frac{\partial u}{\partial t}\right)=-\left\|{\partial_{t} u}\right\|_{L^2}^2 \leq 0, \quad \forall t>0,
\]
where $(\cdot, \cdot)$ represents the standard $L^2$ inner product.

Exponential integrators (EI) have a long history, dating back to early works such as
\cite{MR102923_1958,MR117917_1960,Pope1963ExponentialMethod,Lawson1967GeneralizedRungeKutta,EhleLawson1975GeneralizedRungeKutta,HochbruckEtAl1998ExponentialIntegrators,CoxMatthews2002ExponentialTime}.
The fundamental idea of EI is to perform an exact integral of the linear part while approximating the integral of the nonlinear part.
In particular, when the nonlinear part vanishes, the method reduces to computing the exponential function of the linear operator (or matrix).
Although computing matrix exponential was historically considered impractical, significant advancements have been achieved over the past two decades, particularly in Krylov subspace methods \cite{GallopoulosSaad1992EfficientSolution,HochbruckLubich1997KrylovSubspace,HochbruckEtAl1998ExponentialIntegrators,MolerVanLoan2003NineteenDubious}.
For comprehensive reviews of exponential integrators and their implementation, we direct readers to~\cite{HochOster2010EI,Fasi2024}.

In recent years, exponential integrators have emerged as a powerful tool for developing unconditionally maximum bound preserving (MBP) schemes for the Allen--Cahn equation. The maximum bound principle is physically crucial for Allen--Cahn equations, as its violation can lead to unphysical solutions or numerical instability, particularly in long-time simulations. In 2019, Du et al.~\cite{DuEtAl2019MaximumPrinciple} introduced two unconditionally MBP schemes using the exponential time differencing (ETD) method: the first-order stabilized ETD1 scheme and the second-order stabilized ETD Runge--Kutta (ETDRK2) scheme. Later, Du et al.~\cite{DuEtAl2021MaximumBound} proved that, for a class of semilinear parabolic equations, both stabilized ETD1 and ETDRK2 schemes unconditionally preserve the maximum bound. This approach was further extended to matrix-valued equations~\cite{LiuEtAl2024MaximumBound}. 
However, \cite{DuEtAl2021MaximumBound} asserted that classical ETD methods of order higher than two cannot unconditionally preserve the maximum bound. To address this limitation, Li et al. \cite{LiEtAl2021StabilizedIntegrating} combined the integrating factor Runge–Kutta (IFRK) method, originally introduced in \cite{Lawson1967GeneralizedRungeKutta}, with the stabilization technique proposed in \cite{DuEtAl2021MaximumBound}. This combination led to the development of a class of stabilized IFRK schemes that are unconditionally MBP and achieve up to third-order accuracy. However, in their work~\cite{LiEtAl2021StabilizedIntegrating}, they explicitly acknowledged that they were unable to construct fourth-order or higher-order unconditionally MBP schemes using this approach. To overcome these challenges, various postprocessing techniques have been employed to enforce the maximum bound. For example, a (nonsmooth) cut-off technique was utilized in~\cite{LiEtAl2020ArbitrarilyHighOrder, YangEtAl2022ArbitrarilyHighOrder} to develop arbitrarily high-order MBP multistep ETD schemes. More recently, a (smooth) rescaling technique was proposed in~\cite{QuanEtAlMaximumBound} to ensure both the maximum bound and the original energy dissipation for arbitrarily high-order single-step ETD methods. See also a closely related Lagrangian multiplier approach \cite{ChengShen:CMAME2022, ChengShen:SINUM2022}.  These studies collectively highlight the need for post-processing procedures to preserve the maximum bound. This raises an intriguing question: \textit{is it possible to construct arbitrarily high-order time-stepping schemes that unconditionally preserve the maximum bound without any postprocessing?} This study aims to provide an answer to this question.

In this paper, we develop and analyze a class of time-stepping schemes for the Allen–Cahn equation that achieve arbitrarily high-order temporal accuracy while unconditionally preserving the maximum bound of the phase field variable. These schemes are constructed within the iterative framework of exponential integrators, combined with carefully selected numerical quadrature rules.  The choice of quadrature rule is critical in ensuring that the numerical scheme preserves the maximum bound. Specifically, we employ quadrature rules of the form  
\[
\int_{0}^{\tau} g(s) \, \mathrm{d}s = \sum_{i=1}^{m} w_i g(s_i) + \text{Residue},
\]
such that, for $g(s) = a^s $ with $a > 1 $, the Residue is always nonnegative. In other words, the quadrature rules we choose underestimate the integral for exponential functions. We demonstrate that Gauss–Legendre and left Gauss–Radau quadrature rules satisfy this property while achieving arbitrarily high-order accuracy.  
This underestimation ensures that the numerical solution remains strictly within the theoretical upper bound, thereby preserving the maximum bound at the numerical level. Notably, our proposed schemes are rigorously proven to unconditionally preserve the maximum bound without requiring any additional postprocessing techniques, while simultaneously achieving arbitrarily high-order temporal accuracy. To the best of our knowledge, this is the first MBP numerical method that achieves arbitrarily high-order accuracy for solving the Allen–Cahn equation without relying on postprocessing techniques. Numerical experiments are provided to validate the theoretical results and demonstrate the accuracy and robustness of the proposed methods. Furthermore, in our numerical results, we demonstrate that inappropriate choices of quadrature rules, such as Gauss--Lobatto and right Gauss-Radau quadrature rules, can lead to violations of the maximum bound principle and result in unphysical solution behavior. This observation highlights the critical importance of our quadrature selection strategy. 

The organization of this paper is as follows. In Section \ref{sec:preliminaries}, we reformulate equation \eqref{eq:AC} using the stabilization technique and introduce several useful lemmas for the choice of quadrature rules. Section \ref{sec:EI_methods} presents a class of high-order exponential integrator schemes and proves that these schemes are unconditionally MBP. In Section \ref{sec:error}, we provide a detailed error analysis, demonstrating the optimal convergence of the $k$th-order schemes, where $k$ can be arbitrarily large. Finally, in Section \ref{sec:experiments}, we perform numerical experiments to validate the theoretical results.

\section{Preliminaries}\label{sec:preliminaries}

To begin, we rewrite the Allen--Cahn equation using the  stabilization technique \cite{DuEtAl2021MaximumBound}. 
By introducing a stabilizing constant 
$\kappa > 0$, and adding and subtracting a stabilization term 
$\kappa {u}$ to the equation \eqref{eq:AC}, 
we obtain an equivalent form
\begin{equation}\label{eq:AC_stablized_form}
    u_{t}= \Lk u+\mathcal{N}(u), \quad \bm{x} \in \Omega,~ t>0,
\end{equation}
where the linear operator $\Lk$ and nonlinear operator $\mathcal N$ are defined as
\begin{equation}\label{eq:L_kappa, N}
    \Lk u \coloneqq \varepsilon^2 \Delta u - \kappa u,\quad
    \mathcal{N}(u) \coloneqq f(u)+\kappa u.
\end{equation}
The stabilization constant 
$\kappa$ must satisfy the condition:
\begin{equation}\label{condition_kappa}
    \kappa \ge  \max _{|\xi| \leq \beta}\left|f^{\prime}(\xi)\right|.
\end{equation}
For example, for the Allen--Cahn equation with polynomial potential $f(u)=u-u^3$, it is required that $\kappa \ge 2$ where $\beta=1$.
By the condition \eqref{condition_kappa}, we have
\begin{equation}\label{N_prime}
    0<\hN^{\prime}(\xi)=f^{\prime}(\xi)+\kappa<2\kappa, \quad \forall \xi \in [-\beta, \beta].
\end{equation}
From the assumption \eqref{assump_f} and \eqref{N_prime}, it follows that for any $\xi\in[-\beta,\beta]$,
\begin{equation}\label{N_bound}
    |\hN(\xi)|\leq \kappa\beta,
\end{equation}
because
\begin{equation*}
    -\kappa\beta \le -\kappa\beta + f(-\beta) = \hN(-\beta) \le
    \hN(\xi) \le \hN(\beta) = f(\beta)+\kappa\beta \le \kappa\beta.
\end{equation*}


To establish the preservation of maximum bound, we need the following lemma of Gauss type quadrature rules.



\begin{lemma}\label{lemma:gauss_quadrature}
    For a function \( g \) defined on \( [a,b] \), the optimal quadrature rule of degree \( 2J+K-1 \) for numerical integration is given by
    \begin{equation}
        \int_a^b g(x) \rd x = \sum_{j=1}^J w_j g(t_j) + \sum_{k=1}^{K} v_k g(z_k) + R_{J,K}(g),
    \end{equation}
    where \( \{t_j\}_{j=1}^{J} \subset (a,b) \) are inner nodes, and \( \{z_k\}_{k=1}^{K} \subset \{a,b\} \) are boundary nodes. The corresponding weights \( \{w_j\}_{j=1}^{J} \) and \( \{v_k\}_{k=1}^{K} \) are positive. The remainder term \( R_{J,K}(g) \) vanishes whenever \( g \) is a polynomial of degree at most \( 2J+K-1 \).

    If \( g \) is \( (2J+K) \)-times continuously differentiable on \( [a, b] \), the remainder term is
    \begin{equation}
        R_{J,K}(g) = \frac{g^{(2J+K)}(\xi)}{(2J+K)!} \int_a^b \prod_{k=1}^{K} (x - z_k) \prod_{j=1}^{J} (x - t_j)^2 \rd x, \quad \xi \in [a, b].
    \end{equation}
\end{lemma}

\begin{remark}
The quadrature rule in Lemma~\ref{lemma:gauss_quadrature} is part of the well-known Gauss quadrature family:  
\begin{itemize}
    \item[(i)] For $K = 0$, the rule reduces to the Gauss--Legendre quadrature rule, which does not involve any prescribed boundary nodes.  
    \item[(ii)] For $K = 1$ with $z_1 = a$, the rule corresponds to the left Gauss--Radau quadrature rule.  
    \item[(iii)] For $K = 1$ with $z_1 = b$, it becomes the right Gauss--Radau quadrature rule.  
    \item[(iv)] For $K = 2$ with $z_1 = a$ and $z_2 = b$, the rule corresponds to the Gauss--Lobatto quadrature rule.  
\end{itemize}
Furthermore, if the function $g$ has a positive $(2J+K)$-th derivative over $[a, b]$, the Gauss--Legendre and left Gauss--Radau rules underestimate the exact integral, while the right Gauss--Radau and Gauss--Lobatto rules overestimate it.
\end{remark}

We note that Lemma \ref{lemma:gauss_quadrature} is a direct consequence of the following well-known results; see, e.g., \cite[Theorem 2.1.5.9]{StoerBulirsch2002IntroductionNumerical}.
\begin{lemma} 
    \label{lemma:gauss_quadrature_2}
    Let the real-valued function \( g \) be \( (n+1) \)-times differentiable on the interval \([a, b]\), and consider \( m \) distinct nodes \( x_i \in [a, b] \), where \( x_1 < x_2 < \cdots < x_m \). Suppose the polynomial \( P(x) \) is of degree at most \( n \), satisfies \( \sum_{i=1}^m n_i = n+1 \), and interpolates \( g \) such that  
    \[
    P^{(k)}\left(x_i\right) = g^{(k)}\left(x_i\right), \quad i = 1, 2, \ldots, m, \quad k = 0, 1, \ldots, n_i-1.
    \]  
    Then, for any \( \bar{x} \in [a, b] \), there exists \( \bar{\xi} \in I[x_1, \ldots, x_m, \bar{x}] \) such that  
    \[
    g(\bar{x}) - P(\bar{x}) = \frac{\omega(\bar{x}) g^{(n+1)}(\bar{\xi})}{(n+1)!},
    \]  
    where \( \omega(x) = \prod_{i=1}^m (x - x_i)^{n_i} \), and \( I[x_1, \ldots, x_m, \bar{x}] \) denotes the smallest interval containing \( x_1, \ldots, x_m, \bar{x} \).
\end{lemma}

\begin{proof}[Proof of Lemma \ref{lemma:gauss_quadrature}]
    Consider the polynomial \( P(x) \) of degree at most $2J+K-1$ that interpolates \( g \) such that
    \begin{equation*}
    P(z_k) = g(z_k), \quad
        P(t_j) = g(t_j), \quad
        P^{\prime}(t_j) = g^{\prime}(t_j), \quad 1 \leq j \leq J,\quad 1 \leq k \leq K.
    \end{equation*}
    By Lemma~\ref{lemma:gauss_quadrature_2}, there exists \( \xi \in [a,b] \) such that
    \begin{equation} \label{eq:xi}
        g(x) - P(x) = \frac{g^{(2J+K)}(\xi)}{(2J+K)!} 
        \prod_{k=1}^{K} (x - z_k) \prod_{j=1}^{J} (x - t_j)^2.
    \end{equation}
    Since the polynomial \( P \) is at most of degree \( 2J+K-1 \), the quadrature rule is exact for \( P \), i.e., the remainder \( R_{J,K}(P) = 0 \). 

Now we define
    \[
        I(g) \coloneqq \int_a^b g(x) \, \rd x \quad \text{and}\quad
        I_N(g) \coloneqq \sum_{j=1}^J w_j g(t_j) + \sum_{k=1}^{K} v_k g(z_k).
    \]
 It is straightforward to observe
    \begin{equation}
        I_N(g) = I_N(P) = I(P).
    \end{equation}
    By \eqref{eq:xi}, it follows that
    \begin{align*}
        I(g) - I_N(g) &= I(g) - I(P) = \int_a^b \big( g(x) - P(x) \big) \, \rd x \\
        &= \frac{g^{(2J+K)}(\xi)}{(2J+K)!} 
        \int_a^b \prod_{k=1}^{K} (x - z_k) \prod_{j=1}^{J} (x - t_j)^2 \, \rd x,
    \end{align*}
    which completes the proof.
\end{proof}

\section{MBP exponential integrators}\label{sec:EI_methods}
In this section, we present a class of exponential integrators constructed through an iterative approach combined with numerical quadrature for solving the stabilized Allen--Cahn equation \eqref{eq:AC_stablized_form}. We subsequently  prove that such exponential integrators are unconditionally MBP.

Given a positive integer $N$, we define the uniform time step as $\tau=T/N$ 
and set the time nodes $t_n=n \tau$ for $n=0,1,\cdots,N$.
We focus on the equivalent equation \eqref{eq:AC_stablized_form} on the interval $\left[t_n, t_{n+1}\right]$, or equivalently $w^{n}(\bm{x},s)=u\left(\bm{x},t_n+s\right)$ satisfying the system
\begin{equation}\label{eq:w}
    \left\{
    \begin{aligned}
         & \partial_{s} w^{n}  =\Lk w^{n} +\mathcal{N}(w^{n}), &  & \bm{x} \in \Omega ,~ s \in(0, \tau], \\
         & w^{n}(\bm{x},0)    =u\left(\bm{x},t_{n}\right),     &  & \bm{x} \in \overline{\Omega},
    \end{aligned}
    \right.
\end{equation}
equipped with periodic boundary conditions.
Applying Duhamel's principle to this system, we obtain
\begin{equation}\label{eq:Duhamel_w}
    w^{n}(\bm{x}, \tau) =
    \e^{\tau\Lk} w^{n}(\bm{x},0)
    + \int_0^\tau \e^{(\tau-s)\Lk} \mathcal{N}\left(w^{n}(\bm{x},s)\right)  \rd s.
\end{equation}

The key point in designing a numerical scheme is to find an appropriate way to approximate the nonlinear term $\mathcal{N}(w^{n}(s))$. 
Assume that the numerical solution at time $t_n$,
denoted by $u^{n}$, is already known.
A straightforward approach is to approximate $\mathcal{N}(w^n(s))$ by $\mathcal{N}(u^n)$, which is nothing new but the first-order ETD scheme. We define the solution of first-order ETD scheme as
\begin{equation}\label{eq:w1}
    w^{[1],n}(\tau) \coloneqq \e^{\tau \Lk} u^{n} 
    + (\e^{\tau\Lk}-\hI)\Lk^{-1} \hN(u^{n}).
\end{equation}

Then we use an iterative approach to construct high-order exponential integrators. 
Specifically, for $k \ge 2$, the $k$th-order approximation $w^{[k],n}(\tau)$ is constructed using the $(k-1)$th-order approximation $w^{[k-1],n}(s)$:
\begin{equation}\label{eq:wk1}
    w^{[k],n}(\tau) \approx \e^{\tau \Lk} u^{n}
+ \int_0^\tau \e^{(\tau-s)\Lk} \mathcal{N}(w^{[k-1],n}(s))  \rd s.
\end{equation}
However, the integral in \eqref{eq:wk1} cannot be computed exactly, so we approximate it using a quadrature rule with \( m_k \) nodes:
\begin{equation}\label{k_scheme}
    u^{n+1} = w^{[k],n}(\tau) 
    \coloneqq \e^{\tau \Lk} u^{n}
    + \sum_{i=1}^{m_k} w_i \e^{(\tau-s_i)\Lk}\hN(w^{[k-1],n}(s_i)),
\end{equation} 
where $w_i\ge 0$ and $s_i\in [0,\tau]$ are the weights and nodes of the quadrature rule, respectively. 
To ensure that the scheme achieves \( k \)th-order accuracy, the quadrature rule  
\[
\int_{0}^{\tau} g(s) \rd s \approx \sum_{i=1}^{m_k} w_i g(s_i)
\]  
must have at least \( (k-1) \)-th algebraic accuracy, which introduces a quadrature error of \( O(\tau^{k+1}) \).
Note that $w_i$ and $s_i$ may vary with $k$, as different quadrature rules might be used to achieve higher-order, but we omit this distinction for simplicity.

More precisely, this iterative process is given by as follows:  
For \( j =  2, 3, \dots, k \), the \( j \)th-order approximation \( w^{[j],n}(s) \) is defined as
\begin{equation}\label{eq:wj}
    w^{[j],n}(\tau) 
\coloneqq
\e^{\tau \Lk} u^{n}
+ 
\sum_{i=1}^{m_j} w_i
\e^{(\tau-s_i)\Lk} \mathcal{N}(w^{[j-1],n}(s_i)).
\end{equation}
Each \( w^{[j],n}(s) \) depends on the previous approximation \( w^{[j-1],n}(s) \), ultimately tracing back to the first-order approximation \( w^{[1],n}(s) \), which is defined in \eqref{eq:w1}. 
Similarly, the quadrature rule of $j$-th level iteration 
\[
\int_{0}^{\tau} g(s) \rd s \approx \sum_{i=1}^{m_j} w_i g(s_i)
\]  
must have at least \( (j-1) \)-th algebraic accuracy, which introduces a quadrature error of \( O(\tau^{j+1}) \).
We have developed a systematic iterative framework for deriving arbitrarily high-order  exponential integrators. These integrators achieve $k$th-order accuracy by approximating the integral term with suitably chosen quadrature rules. 
For convenience, we denote the $k$th-order exponential integrators as EI$k$.

The following theorem establishes the discrete maximum principle of EI$k$ \eqref{k_scheme}.
\begin{theorem}[discrete maximum principle]\label{thm:mp}
    Suppose that Assumption \ref{assump_f} holds and the stabilization constant $\kappa$ satisfies \eqref{condition_kappa}.
    If the Gauss--Legendre or left Gauss--Radau quadrature rule is applied at each iteration level in \eqref{eq:wj}, then the $k$th-order exponential integrators \eqref{k_scheme} unconditionally preserve the discrete maximum principle. 
    More precisely, let $\{u^n\}$ be the solution of the $k$th-order exponential integrators \eqref{k_scheme} with a time step size $\tau >0$.
    If the initial data satisfies $\left\| u^{0} \right\|_{\Linf} \le \beta$, then $\left\|u^{n}\right\|_{\Linf} \le \beta$ for any $n\ge 1$ and any $\tau >0$.
\end{theorem}
\begin{proof}
It suffices to prove that $\left\|u^{n}\right\|_{\Linf} \le \beta$ implies $\left\|u^{n+1}\right\|_{\Linf} \le \beta$ for any $n$.
First, we note that the family of operators $\{ e^{t \Delta} \}_{t\ge0}$ forms a strongly continuous semigroup on $ L^{\infty}(\Omega)$  and the parabolic maximum principle implies
\[
\left\| e^{t \Delta} u \right\|_{L^{\infty}} \leq \| u \|_{L^{\infty}}, \quad \forall~ t \geq 0,~ u \in L^\infty(\Omega),
\]
which shows that $e^{t \Delta}$ is a contraction map in $L^\infty(\Omega)$; see, e.g., \cite[Chapter 8]{Jost2013PartialDifferential}.

    By \eqref{N_bound}, we obtain $\left\| \hN(u^{n}) \right\|_{\Linf}\le \kappa\beta$ from $\left\| u^{n} \right\|_{\Linf} \le \beta$.
    Then, for the first-order scheme, we have
    \begin{align*}
        \left\| w^{[1],n}(\tau) \right\|_{\Linf}  
        &= 
        \left\| \e^{-\kappa\tau}\e^{\tau \varepsilon^2 \Delta} u^{n} 
        + \int_0^\tau  \e^{-\kappa(\tau-s)} \e^{(\tau-s) \varepsilon^2 \Delta} \mathcal{N}(u^{n}) \rd s \right\|_{\Linf} 
        \\ 
        &\le
        \e^{-\kappa\tau} \left\| u^{n} \right\|_{\Linf} 
        + \int_0^\tau \e^{-\kappa(\tau-s)} \left\| \mathcal{N}(u^{n}) \right\|_{\Linf}  \rd s 
        \\
        &\le  \e^{-\kappa\tau} \beta + \kappa\beta \int_0^\tau \e^{-\kappa(\tau-s)} \rd s = \beta.
    \end{align*}

    For $2\le j\le k$, assume that the $(j-1)$th-order scheme preserves the maximum bound unconditionally, i.e., $\left\| w^{[j-1],n}(\tau) \right\|_{\Linf} \le \beta$ for any $\tau >0$.
    Then, by \eqref{N_bound}, it follows that $\left\| \hN(w^{[j-1],n}(\tau)) \right\|_{\Linf}\le \kappa\beta$ for any $\tau >0$.
    Similarly, for the $j$th-order scheme, we have
    \begin{equation}\label{ineq:mbp1}
    \begin{aligned}
    \left\| w^{[j],n}(\tau) \right\|_{\Linf}
    &= \left\| \e^{\tau \Lk} u^{n} + \sum_{i=1}^{m_j} w_i \e^{(\tau-s_i)\Lk}\hN(w^{[j-1],n}(s_i)) \right\|_{\Linf}
    \\
    &\le \e^{-\kappa\tau} \left\| u^{n} \right\|_{\Linf} + \sum_{i=1}^{m_j} w_i \e^{-\kappa(\tau-s_i)} \left\| \hN(w^{[j-1],n}(s_i)) \right\|_{\Linf}
    \\
    &\le  \e^{-\kappa\tau}\beta + \kappa\beta\sum_{i=1}^{m_j} w_i \e^{\kappa (s_i-\tau)}
    = e^{-\kappa\tau} \left(1+\kappa\sum_{i=1}^{m_j} w_i \e^{\kappa s_i}\right)\beta. 
    \end{aligned}
    \end{equation}
    According to Lemma \ref{lemma:gauss_quadrature}, 
    for the Gauss quadrature rule, the remainder is
        \begin{equation*}
        \int_{0}^{\tau} \e^{\kappa s} \rd s - \sum_{i=1}^{m_j} w_i \e^{\kappa s_i}
        = \frac{\e^{\kappa \xi_1}}{\kappa^{2m_j}(2m_j)!}\int_0^{\tau} \prod_{j=1}^{m_j}(s-s_i)^2 \rd s \ge 0.
    \end{equation*}
    For left Gauss--Radau quadrature rule, the remainder is
    \begin{equation*}
        \int_{0}^{\tau} \e^{\kappa s} \rd s - \sum_{i=1}^{m_j} w_i \e^{\kappa s_i}
        = \frac{\e^{\kappa \xi_2}}{\kappa^{2m_j+1}(2m_j+1)!}\int_0^{\tau} s\prod_{j=1}^{m_j}(s-s_i)^2 \rd s 
        \ge 0.
    \end{equation*}
    Since the remainder term is non-negative for both the Gauss and the left Gauss--Radau quadrature rules, we obtain
    \begin{equation}\label{ineq:mp_3}
      \kappa\sum_{i=1}^{m_j} w_i \e^{\kappa s_i} \le \kappa \int_{0}^{\tau} \e^{\kappa s} \rd s = \e^{\kappa\tau}-1. 
    \end{equation}
Substituting \eqref{ineq:mp_3} into \eqref{ineq:mbp1}, we obtain  
$$
\left\| w^{[j],n}(\tau) \right\|_{\Linf} \le \beta, \quad \forall \tau > 0.
$$

    By induction, we have
    \[
        \left\| u^{n+1} \right\|_{\Linf}=\left\| w^{[k],n}(\tau) \right\|_{\Linf} \le \beta, \quad \forall \tau >0,
    \]
    which completes the proof.
\end{proof}

\begin{remark}
From the proof of Theorem \ref{thm:mp}, we observe that the key to preserve the maximum bound lies in selecting a quadrature rule such that the numerical approximation of $\e^x$ does not exceed its exact integral. Notably, the Gauss--Legendre quadrature rule and the left Gauss--Radau quadrature rule satisfy this requirement and are therefore suitable for constructing maximum bound preserving schemes. In contrast, the right Gauss--Radau quadrature rule and the Gauss--Lobatto quadrature rule fail to meet this requirement, which may result in numerical schemes that violate the maximum principle and exhibit incorrect dynamics. This behavior is observed in numerical experiments, as shown in Figures \ref{fig:dynamics}, \ref{fig:blowup}, and \ref{fig:lGR}.
\end{remark}

\section{Error analysis}\label{sec:error}
This section focuses on the error analysis of the numerical scheme \eqref{eq:wk1}-\eqref{eq:wj}, carried out within the iterative framework that was also used for its construction.
\begin{theorem}
    Suppose that Assumption \ref{assump_f} holds and the stabilization constant $\kappa$ satisfies \eqref{condition_kappa}.
    Let $k>0$ be the order and $T>0$ be a given total time. Suppose that $u(t)$ is the exact solution of \eqref{eq:AC} with initial data satisfying  
    \(
    \left\|u^{0}\right\|_{L^{\infty}}\le \beta.
    \)
    Additionally, assume that { $u\in C^1([0,T];L^2(\Omega)$ and the nonlinear term $\hN(u)$} satisfies  
\begin{equation}\label{cond_f_regularity}
        \sup_{0<t<T} \|\Delta^j \hN(u(t))\|_{L^2(\Omega)} < \infty, \quad \forall j \leq k.
\end{equation}
    Let $\{u^{n}\}$ be the numerical solution obtained by the $k$th-order scheme \eqref{k_scheme} with a time step size $\tau >0$.  
    Then, for any $\tau < 1$, there exists a constant $C>0$ independent of $\tau$ such that, at $t_n = n\tau \leq T$, the error estimate holds:
    \[
    \left\| u(t_n)-u^{n} \right\|_{L^2} \le C \e^{2\kappa T} \tau^k.
    \]
\end{theorem}
\begin{proof}
    Denote the error at time \( t_n \) as \( e_n \coloneqq u^n - u(t_n) \). Further, define the error function \( e^{[k],n}(s) \coloneqq w^{[k],n}(s) - u(t_n + s) \). Note that \( e^{[k],n}(0) = e_n \) and \( e^{[k],n}(\tau) = e_{n+1} \).

\textbf{Step 1: error estimate for the first iteration.}    For the first-order scheme \eqref{eq:w1}, the error is expressed as:  
\begin{equation}\label{eq:e_1}
    \begin{aligned}
        e^{[1],n}(\tau) 
        &= \e^{\tau \Lk} e_n + \int_0^\tau \e^{(\tau-s)\Lk} \left(\hN(u^n) - \hN(u(t_n + s))\right) \rd s \\
        &= \e^{\tau \Lk} e_n + \int_0^\tau \e^{(\tau-s)\Lk} R_1 \rd s + \int_0^\tau \e^{(\tau-s)\Lk} R_2 \rd s,
      \end{aligned}
\end{equation}
    where \( R_1 = \hN(u^n) - \hN(u(t_n)) \) and \( R_2(s) = \hN(u(t_n)) - \hN(u(t_n + s)) \).
    Using \eqref{N_prime}, we know
\begin{equation}\label{Nxi1-Nxi2}
        |\hN(\xi_1) - \hN(\xi_2)| \leq 2\kappa |\xi_1 - \xi_2|, \quad \forall \, \xi_1, \xi_2 \in [-\beta, \beta].
\end{equation}
Since $\| u^n \|_{\Linf} \leq \beta$, as established in Theorem~\ref{thm:mp}, we obtain 
    \begin{equation}\label{ineq:R1R2}
        \left\| R_1 \right\|_{L^2} \leq 2\kappa \left\| e_n \right\|_{L^2}, \quad 
        \left\| R_2 \right\|_{L^2} \leq 2\kappa \left\| u(t_n) - u(t_n + s) \right\|_{L^2} \leq 2\kappa \tilde{C}_1 s,
    \end{equation}  
    where \( \tilde{C}_1 \coloneqq \sup_{t \in [0,T]} \left\| u^{\prime}(t) \right\|_{L^2(\Omega)} \).
Using the inequality $\left\| \e^{t\Lk} v \right\|_{L^2} \le \left\| v \right\|_{L^2}$ for any $t > 0$, and the following triangle inequality  
    {\[ \left\| \int_0^\tau g(\bx,t) \rd t \right\|_{L^2} \leq  \int_0^\tau \| g(\bx,t) \|_{L^2} \rd t, \]}
we substitute \eqref{ineq:R1R2} into \eqref{eq:e_1}, giving
\begin{equation}\label{eq:e_1_final}
    \begin{aligned}
        \left\| e^{[1],n}(\tau) \right\|_{L^2} 
        &\le \left\| \e^{\tau\Lk} e_n \right\|_{L^2} + \left\| \int_0^\tau \e^{(\tau-s)\Lk} R_1 \rd s \right\|_{L^2} + \left\| \int_0^\tau \e^{(\tau-s)\Lk} R_2 \rd s \right\|_{L^2} \\
        &\le \left\| e_n \right\|_{L^2} 
        +   \int_0^\tau \left\|  R_1 \right\|_{L^2} \rd s   
        +  \int_0^\tau \left\|  R_2 \right\|_{L^2} \rd s    \\
        &\le (1 + 2\kappa\tau) \left\| e_n \right\|_{L^2} + {\kappa \tilde{C}_1} \tau^2 = (1 + 2\kappa\tau) \left\| e_n \right\|_{L^2} + C_1 \tau^2,
    \end{aligned}
\end{equation}
    where \( C_1 \coloneqq \kappa \tilde{C}_1 \).

\textbf{Step 2: error estimate from \( (j-1) \)-level iteration to \( j \)-level iteration.} For $2\le j\le k$, suppose that 
\begin{equation}\label{inductive_hypothesis}
      \left\| e^{[j-1],n}(\tau) \right\|_{L^2} 
      \le  
      \left[\sum_{\ell=0}^{j-1}\frac{(2\kappa\tau)^\ell}{\ell!}\right] \left\|e_{n}\right\|_{L^2}
      + \left[\sum_{\ell=1}^{j-1}\frac{(\ell+1)!(2\kappa)^{j-1-\ell}}{j!}C_\ell\right] \tau^{j}, 
\end{equation}
    holds for any $\tau > 0$, where $C_{\ell}$ are constants independent of $\tau$ and $\kappa$.
    Note that when $j=2$, the inequality \eqref{inductive_hypothesis} is equivalent to the inequality \eqref{eq:e_1_final} derived above.
    Then, we consider the $j$th level iteration.
    We introduce a function $v^{[j],n}(s)$ between $u(t_n+s)$ and $w^{[j],n}(s)$ as
\begin{equation}\label{def:v_j}
    v^{[j],n} \coloneqq  \e^{\tau \Lk} u(t_n)
    + \sum_{i=1}^{m_j} w_i \e^{(\tau-s_i)\Lk}\hN(u(t_n+s_i))
\end{equation}
    and denote
    \[
    \eta^{[j],n}(s) \coloneqq w^{[j],n}(s) - v^{[j],n}(s), \quad
    \xi^{[j],n}(s) \coloneqq v^{[j],n}(s) - u(t_n+s),
    \]
    which imply the error decomposition
\begin{equation}\label{eq:err_decomp}
    e^{[j],n}(s) = \eta^{[j],n}(s) + \xi^{[j],n}(s).
\end{equation}

\textbf{Step 2.1: estimate for \( \eta^{[j],n} \).} To estimate $\eta^{[j],n}$, considering the difference between \eqref{k_scheme} and \eqref{def:v_j}, we have
    \begin{equation}\label{eq:eta_j}            
    \begin{aligned}
      \left\| \eta^{[j],n}(\tau)   \right\|_{L^2}
      & = \left\| \e^{\tau\Lk}e_{n}+\sum_{i=1}^{m_j}w_i\e^{(\tau-s_i)\Lk}\left[ \hN(w^{[j-1],n}(s_i)) - \hN(u(t_n+s_i)) \right]  \right\|_{L^2} \\
      & \le \left\| e_{n} \right\|_{L^2} + \sum_{i=1}^{m_j}w_i\left\| \hN(w^{[j-1],n}(s_i)) - \hN(u(t_n+s_i)) \right\|_{L^2},
    \end{aligned}
    \end{equation}
    where we have use the property that quadrature weights $w_{i}$ are positive.
    By the discrete maximum principle Theorem~\ref{thm:mp}, we know $\left| w^{[j-1],n}(\bx,s) \right| \le \beta$ for any $s>0$ and $\bx \in \Omega$. Then, by \eqref{Nxi1-Nxi2}, we obtain
\begin{equation}\label{ineq:R_3}
    \left\|  \hN(w^{[j-1],n}(s_i)) - \hN(u(t_n+s_i)) \right\|_{L^2} 
    \le 
    2\kappa \left\| e^{[j-1],n}(s_i) \right\|_{L^2}.
\end{equation}
Putting \eqref{ineq:R_3} in \eqref{eq:eta_j} and using the inductive hypothesis \eqref{inductive_hypothesis}, we obtain

    \begin{equation}\label{eq:eta_j_2}            
    \begin{aligned}
      \left\| \eta^{[j],n}(\tau)   \right\|_{L^2}
      & \le \left\| e_{n} \right\|_{L^2} \\
      &\quad + 2\kappa\sum_{i=1}^{m_j}w_i \left\{ \left[\sum_{\ell=0}^{j-1}\frac{(2\kappa s_i)^\ell}{\ell!}\right] \left\|e_{n}\right\|_{L^2}
      + \left[\sum_{\ell=1}^{j-1}\frac{(\ell+1)!(2\kappa)^{j-1-\ell}}{j!}C_\ell\right] s_i^{j} \right\}. 
    \end{aligned}
    \end{equation}
    Note that the quadrature rule has at least $(j-1)$th-order algebraic accuracy, which implies
    \begin{equation}\label{q_1}
      \sum_{i=1}^{m_j}w_i s_i^{\ell} = \int_{0}^{\tau} s^{\ell} \mathrm{d} s = \frac{\tau^{\ell+1}}{\ell+1}, \quad \ell=0,1,\cdots,j-1.
    \end{equation}
    If the quadrature rule has just $(j-1)$th-order accuracy, 
    since we choose the Gauss rule or the left Gauss--Radau rule, by Lemma~\ref{lemma:gauss_quadrature}, we know that the reminder has the sign of the $j$th derivative of the integrand $s^j$, which is positive in $(0,\tau)$. Therefore, we have
    \begin{equation}\label{q_3}
      \sum_{i=1}^{m_j}w_i s_i^{j} \le \int_{0}^{\tau} s^{j} \mathrm{d} s = \frac{\tau^{j+1}}{j+1}.
    \end{equation}
    If the quadrature rule has $j$th-order or higher-order accuracy, then inequality~\eqref{q_3} attains equality.
Substituting \eqref{q_1} and \eqref{q_3} into \eqref{eq:eta_j_2}, we obtain
    \begin{equation}\label{eta_j_final}
        \begin{aligned}
            \left\| \eta^{[j],n}(\tau) \right\|_{L^2}
            & \le \left[1 + 2\kappa\tau\sum_{\ell=0}^{j-1}\frac{(2\kappa)^\ell}{\ell!} \sum_{i=1}^{m_j}w_i s_{i}^{\ell}  \right] \left\|e_{n}\right\|_{L^2}
            + 2\kappa\tau \left[\sum_{\ell=1}^{j-1}\frac{(\ell+1)!(2\kappa)^{j-1-\ell}}{j!}C_\ell\right] \sum_{i=1}^{m_j}w_i s_i^{j} \\
            &\le \left[\sum_{\ell=0}^{j}\frac{(2\kappa\tau)^\ell}{\ell!}\right] \left\|e_{n}\right\|_{L^2}
            + \left[\sum_{\ell=1}^{j-1}\frac{(\ell+1)!(2\kappa)^{j-\ell}}{(j+1)!}C_\ell\right] \tau^{j+1}. 
        \end{aligned}
    \end{equation}

\textbf{Step 2.2: estimate for \( \xi^{[j],n} \).}  
To derive an estimate for the term $\xi^{[j],n}$, we consider the difference between \eqref{eq:Duhamel_w} and \eqref{def:v_j}:
    \begin{equation}
    \begin{aligned}
      \left\| \xi^{[j],n}(\tau) \right\|_{L^2} 
      & = \left\|  \sum_{i=1}^{m_j} w_i \e^{(\tau-s_i)\Lk}\hN(u(t_n+s_i)) - \int_{0}^{\tau} \e^{(\tau-s)\Lk} \hN(u(t_n+s)) \mathrm{d}x \right\|_{L^2}.
    \end{aligned}
    \end{equation}
    By Lemma~\ref{lemma:gauss_quadrature}, the quadrature error of a $(j-1)$th-order accurate rule is $O(\tau^{j+1})$, provided that the $j$th derivative of the integrand is bounded. This condition is ensured by \eqref{cond_f_regularity}. More precisely, there exists a constant $C_j$ such that
    \begin{equation}\label{xi_j_final}
    \left\| \xi^{[j],n}(\tau) \right\|_{L^2} \le C_j \tau^{j+1}.
    \end{equation}
    Furthermore, if a quadrature rule of $j$th-order or higher-order accuracy is used, the quadrature error remains $O(\tau^{j+1})$, provided that $\tau < 1$.

Combining \eqref{eq:err_decomp}, \eqref{eta_j_final} and \eqref{xi_j_final},  we obtain
\begin{equation}\label{eq:e_j_final}
    \left\| e^{[j],n}(\tau) \right\|_{L^2} \le \left[\sum_{\ell=0}^{j}\frac{(2\kappa\tau)^\ell}{\ell!}\right] \left\|e_{n}\right\|_{L^2}
    + M_j \tau^{j+1},
\end{equation}
where
\begin{equation*}
    M_j \coloneqq \sum_{\ell=1}^{j}\frac{(\ell+1)!(2\kappa)^{j-\ell}}{(j+1)!}C_\ell.
\end{equation*}

    \textbf{Step 3: global error estimate.}     
By induction, the inequality \eqref{eq:e_j_final} holds for any $j\ge 1$. Let $j=k$, then we have
\begin{equation*}
\begin{aligned}
    \left\| e_{n+1} \right\|_{L^2} 
    &= \left\| e^{[k],n}(\tau) \right\|_{L^2}  \le \left[\sum_{\ell=0}^{k}\frac{(2\kappa\tau)^\ell}{\ell!}\right] \left\|e_{n}\right\|_{L^2}
    + M_k\tau^{k+1}\\
    &\le \left[\sum_{\ell=0}^{k}\frac{(2\kappa\tau)^\ell}{\ell!}\right]^{n} \left\|e_{0}\right\|_{L^2} + M_k\tau^{k+1} \sum_{m=0}^{n-1} \left[\sum_{\ell=0}^{k}\frac{(2\kappa\tau)^\ell}{\ell!}\right]^{m}\\
    & = M_k \tau^{k+1} \frac{\left[\sum_{\ell=0}^{k} \frac{(2\kappa\tau)^\ell}{\ell!}\right]^{n}-1}{\left[\sum_{\ell=0}^{k}\frac{(2\kappa\tau)^\ell}{\ell!}\right]-1}\\ 
    &\le M_k \tau^{k+1} \frac{\left[\sum_{\ell=0}^{k} \frac{(2\kappa\tau)^\ell}{\ell!}\right]^{n}}{2\kappa\tau} \le  C \e^{2\kappa t_n} \tau^k,
\end{aligned}
\end{equation*}
where the constant $C = {M_k}/{2\kappa}$. This completes the proof of the theorem.
\end{proof}

\section{Numerical experiments}\label{sec:experiments}
In this section, we carry out some numerical experiments to demonstrate the performance of exponential integrators.
We consider a rectangular domain discretized using a uniform rectangular mesh with a uniform mesh size. The Laplace operator is discretized using the central finite difference method. In this setup, the product of the exponential of the discrete Laplace operator, $\Delta_h$, and a vector can be efficiently computed using the fast Fourier transform (FFT); for implementation details, we refer the reader to~\cite{JuEt2015fast}. 
Since $\e^{t\Delta_h}$ also constitutes a contraction semigroup, the discrete maximum principle remains valid when $\Delta$ is replaced with $\Delta_h$ in Theorem \ref{thm:mp}.

\begin{table}[!htbp]
    \centering
    \renewcommand{\arraystretch}{1.3} 
    \begin{tabular}{|c|c|c|c|}
    \hline
    Number of nodes, $n$ & Nodes, $x_i$ & Weights, $w_i$ & Algebraic accuracy  \\ \hline
    \multirow{2}{*}{2}   & -1 & 1/2 & \multirow{2}{*}{2}    \\ \cline{2-3} 
     & 1/3 & 3/2 & \\ \hline
     \multirow{3}{*}{3}    & -1 & 2/9 & \multirow{3}{*}{4} \\ \cline{2-3} 
     & $(1-\sqrt{6})/5$ &  $(16+\sqrt{6})/18$ & \\ \cline{2-3} 
     & $(1+\sqrt{6})/5$ &  $(16-\sqrt{6})/18$ & \\ \hline
    \end{tabular}
    \caption{The nodes and weights of left Gauss--Radau quadrature rules on the interval $[-1,1]$.}
    \label{tab:GR_quadrature}
\end{table}

To conduct a specific exponential integrator, we need to specify the quadrature rule. 
To compute 
$u^{n+1}$ from $u^{n}$, we use the known value of $u^{n}$ and apply the left Gauss--Radau quadrature rules in the experiments below. 
As shown in Table~\ref{tab:GR_quadrature}, to achieve sufficient accuracy, it is adequate to use the left Gauss--Radau quadrature rule with $2$ nodes for EI2 and EI3, and $3$ nodes for EI4 and EI5. For completeness, we provide below the formulas for EI1 to EI5 (with left Gauss--Radau rules):

    \begin{align}
        w^{[1],n}(\tau) &= \e^{\tau \Lk} u^{n} 
    + (\e^{\tau\Lk}-\hI)\Lk^{-1} \hN(u^{n}),\\
    \label{ex_scheme_2}    w^{[2],n}(\tau) & = \e^{\tau \Lk} u^{n}
    + \frac{\tau}{4} \e^{\tau\Lk}\hN(u^n) 
    + \frac{3\tau}{4} \e^{\frac{1}{3}\tau\Lk}\hN\left(w^{[1],n}\left(\frac{2}{3}\tau\right)\right),  \\
        w^{[3],n}(\tau) & = \e^{\tau \Lk} u^{n}
    + \frac{\tau}{4} \e^{\tau\Lk}\hN(u^n) 
    + \frac{3\tau}{4} \e^{\frac{1}{3}\tau\Lk}\hN\left(w^{[2],n}\left(\frac{2}{3}\tau\right)\right),  \\
    w^{[4],n}(\tau) & = \e^{\tau \Lk} u^{n}
    + \frac{\tau}{9} \e^{\tau\Lk}\hN(u^n) 
    + \tau\sum_{i=1}^{2} c_i  \e^{(1-\theta_i)\tau\Lk}\hN\left(w^{[3],n}\left(\theta_i\tau\right)\right), \\
    w^{[5],n}(\tau) & = \e^{\tau \Lk} u^{n}
    + \frac{\tau}{9} \e^{\tau\Lk}\hN(u^n) 
    + \tau\sum_{i=1}^{2} c_i  \e^{(1-\theta_i)\tau\Lk}\hN\left(w^{[4],n}\left(\theta_i\tau\right)\right),
    \end{align}
where
\[
c_i= \frac{(16\pm \sqrt{6})}{36},\quad \theta_i = \frac{6\mp \sqrt{6}}{10}, \quad i=1,2.
\]

\subsection{Convergence tests}

To validate the accuracy of exponential integrators, we consider the Allen–Cahn equation:  
\begin{equation}\label{eq:AC_cubic}  
    \partial_t u = \varepsilon^2 \Delta u + u - u^3,  
\end{equation}  
on the domain \( \Omega = (0, 2\pi)^2 \) with \( \varepsilon = 0.1 \). The equation is subject to periodic boundary conditions, and the initial condition is given by \( u^0 = 0.1 \sin x \sin y \). We set the stabilization constant to \( \kappa = 2 \) and discretize the domain using a uniform mesh with mesh size \( h = 2\pi / 256 \). 
The maximum bound of this model is $\beta = 1$, and the stabilization constant is chosen as $\kappa = 2$.

The solution is computed at \( T = 1 \) using EI2 to EI5 schemes with time step sizes \( \tau = 0.1 \times 2^{-n} \) for \( n = 0, 1, \dots, 5 \). To evaluate the convergence order of each scheme, we compute relative errors in the \( L^2 \) and \( L^{\infty} \) norms. The results, presented in Table \ref{tab:convergence}, include convergence rates calculated based on errors at different time step sizes.  
As shown in Table \ref{tab:convergence}, the convergence rates for all schemes approach their theoretical values as \( \tau \) decreases.  
\begin{table}[!htbp]
    \centering
    \renewcommand{\arraystretch}{1.2} 
    \begin{tabular}{|c|c|c|c|c|c|}
    \hline
     Method & $\tau$ & $L^2$ Error & Rate & $L^{\infty}$ Error & Rate \\
    \hline
    \multirow{5}{*}{EI2} 
    &$\delta=0.1$  & 2.8404e-01 & - & 2.0802e-03 & - \\
    &$\delta /2$ & 7.8201e-02 & 1.86 & 5.7113e-04 & 1.86 \\
    &$\delta /4$ & 2.0500e-02 & 1.93 & 1.4954e-04 & 1.93 \\
    &$\delta /8$ & 5.2467e-03 & 1.97 & 3.8252e-05 & 1.97 \\
    &$\delta /16$ & 1.3271e-03 & 1.98 & 9.6729e-06 & 1.98 \\
    \hline
    \multirow{5}{*}{EI3}
&$\delta=0.1$   & 2.9742e-02 & - & 2.1523e-04 & - \\
&$\delta /2$  & 4.1182e-03 & 2.85 & 2.9741e-05 & 2.86 \\
&$\delta /4$  & 5.4175e-04 & 2.93 & 3.9087e-06 & 2.93 \\
&$\delta /8$  & 6.9471e-05 & 2.96 & 5.0099e-07 & 2.96 \\
&$\delta /16$ & 8.7955e-06 & 2.98 & 6.3414e-08 & 2.98 \\
    \hline
    \multirow{5}{*}{EI4}
&$\delta=0.1$   & 1.8622e-03 & - & 1.3290e-05 & - \\
&$\delta /2$  & 1.2956e-04 & 3.85 & 9.2309e-07 & 3.85 \\
&$\delta /4$  & 8.5443e-06 & 3.92 & 6.0825e-08 & 3.92 \\
&$\delta /8$  & 5.4856e-07 & 3.96 & 3.9034e-09 & 3.96 \\
&$\delta /16$ & 3.4749e-08 & 3.98 & 2.4721e-10 & 3.98 \\
    \hline
    \multirow{5}{*}{EI5}
&$\delta=0.1$    & 9.5065e-05 & - & 6.6845e-07 & - \\
&$\delta /2$ & 3.3225e-06 & 4.84 & 2.3325e-08 & 4.84 \\
&$\delta /4$ & 1.0981e-07 & 4.92 & 7.7030e-10 & 4.92 \\
&$\delta /8$ & 3.5300e-09 & 4.96 & 2.4751e-11 & 4.96 \\
&$\delta /16$ & 1.1184e-10 & 4.98 & 7.8293e-13 & 4.98 \\
    \hline
    \end{tabular}
    \caption{Temporal convergence rates and errors for EI2 to EI5 methods.}
    \label{tab:convergence}
\end{table}

\subsection{Maximum bound preservation tests}

We consider the Allen--Cahn equation with Flory--Huggins potential
\begin{equation}\label{eq:AC_log}
\partial_t u = \varepsilon^2 \Delta u + f_{\text{FH}},
\end{equation}
where
\begin{equation}\label{f_FH}
    f_{\text{FH}} = \frac{\theta}{2} \ln \frac{1-u}{1+u} + \theta_c u.
\end{equation}
The domain is $\Omega = (0, 2\pi)^2$, and the parameters are set to $\varepsilon = 0.1$, $\theta = 0.8$, and $\theta_c = 1.6$. The equation is subject to periodic boundary conditions.
In this case, preserving the maximum principle is crucial, as the logarithmic term in the equation can lead to numerical instability if the maximum principle is violated.
The maximum bound $\beta$ of this model is the positive root of $f_{\text{FH}}=0$, and the stabilization constant is chosen as 
$$\kappa = \theta/(1-\beta^2)-\theta_c \ge\max_{|\xi| \leq \beta}\left|f_{\text{FH}}^{\prime}(\xi)\right|.$$
In this case, $\beta \approx 0.9575$ and $\kappa = 8.02$.

The spatial domain is discretized using a uniform mesh with a mesh size of $h = 2\pi / 256$. We consider a total simulation time of $T = 50$ and solve the equation using the EI2 to EI5 schemes with left Gauss--Radau rules and time step sizes $\tau = 0.01$ and $\tau = 0.1$. The initial condition is randomly distributed in the range $[-0.8, 0.8]$.

Figures \ref{fig:log_energy_MP_01} and \ref{fig:log_energy_MP_001} present the results, where the energy and supremum norms of the numerical solutions are plotted over time. The supremum norm of the numerical solution is bounded by $\beta$ for all schemes, confirming that the EI2 to EI5 schemes preserve the maximum principle for the Allen--Cahn equation. Additionally, the system's energy decreases over time, and lower-order schemes exhibit time delay effects.

Finally, we perform an experiment to compare the performance of MBP and non-MBP schemes. Specifically, we consider the Allen–Cahn equation with a different scaling law:
\begin{equation}\label{eq:AC_log_2}
\partial_t u = \varepsilon \Delta u + \frac{1}{\varepsilon} f_{\text{FH}},
\end{equation}  
on the domain \( \Omega = (-1, 1)^2 \), subject to periodic boundary conditions, with \( \varepsilon = 0.1 \).
The nonlinear function $f_{\text{FH}}$ is defined in \eqref{f_FH}, and the parameters are set to  \( \theta = 0.68 \) and \( \theta_c = 2.21 \).   
The initial condition is given by  
\begin{equation}
\begin{aligned} 
u^0= & -\beta  \tanh \left[\left(x^2+(y-0.3)^2-0.29^2\right) / \epsilon^2\right] \tanh \left[\left(x^2+(y+0.3)^2-0.29^2\right) / \epsilon^2\right],
\end{aligned}
\end{equation}  
where the maximum bound $\beta$ is the positive root of $f_{\text{FH}}=0$. The stabilization constant is chosen as  $\kappa = [\theta/(1-\beta^2)-\theta_c]/\varepsilon$.

The spatial domain is discretized using a uniform \( 128 \times 128 \) grid. We compute numerical solutions using EI2 to EI4 with both left and right Gauss--Radau quadrature rules. Methods of the same order use the same number of quadrature nodes, and all methods employ a uniform time step size of \( \tau = 10^{-3} \). Additionally, we compute a reference solution using EI5 with the left Gauss--Radau quadrature rule and a smaller time step size of \( \tau = 10^{-4} \).
The terminal time is set as $T=1.5$.


Figure~\ref{fig:dynamics} shows snapshots of numerical solutions at different time points. The solutions computed using EI2 and EI3 with the right Gauss--Radau rule blow up rapidly, as evidenced by the evolution of their supremum norms in Figure~\ref{fig:blowup}. As discussed earlier, the right Gauss–Radau rule does not satisfy the requirement for maximum bound preservation. In contrast, solutions obtained using EI2 to EI4 with the left Gauss--Radau rule strictly preserve the maximum bound principle and exhibit correct dynamics compared to the reference solutions. This aligns with our theoretical results and highlights the critical importance of selecting appropriate quadrature rules. As shown in Figure~\ref{fig:dynamics}, among the EI$k$ methods using the right Gauss--Radau rule, only EI4 avoids blow-up and maintains accurate dynamics, possibly due to its higher accuracy.

Figure~\ref{fig:lGR} illustrates the evolution of energy and supremum norms for solutions that do not blow up. The energy of all these solutions decreases over time. In the right subfigure of Figure~\ref{fig:lGR}, the supremum norm of solutions computed by EI4 with the right Gauss--Radau rule slightly exceeds the maximum bound $\beta$, but this exceed part is minimal due to its high accuracy, allowing the solution to remain stable without blowing up. 


\begin{figure}[htbp]
    \centering
    \begin{subfigure}[b]{0.48\textwidth}
        \centering
        \includegraphics[width=\textwidth]{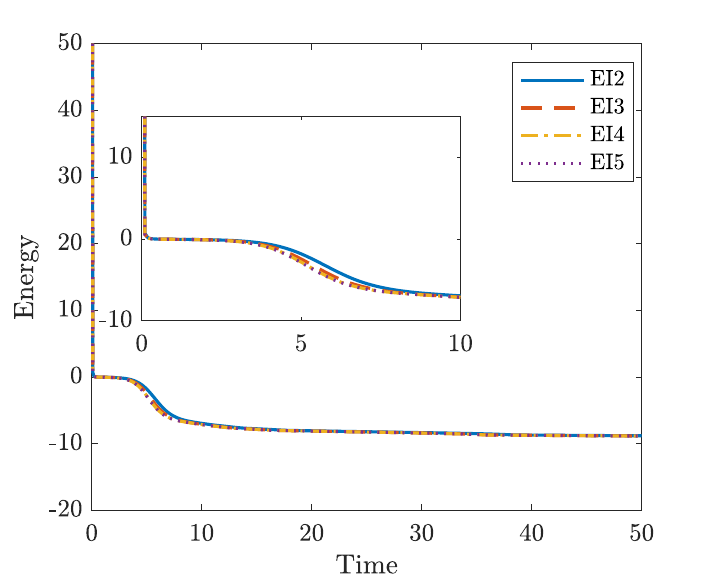}
    \end{subfigure}
    \hfill
    \begin{subfigure}[b]{0.48\textwidth}
        \centering
        \includegraphics[width=\textwidth]{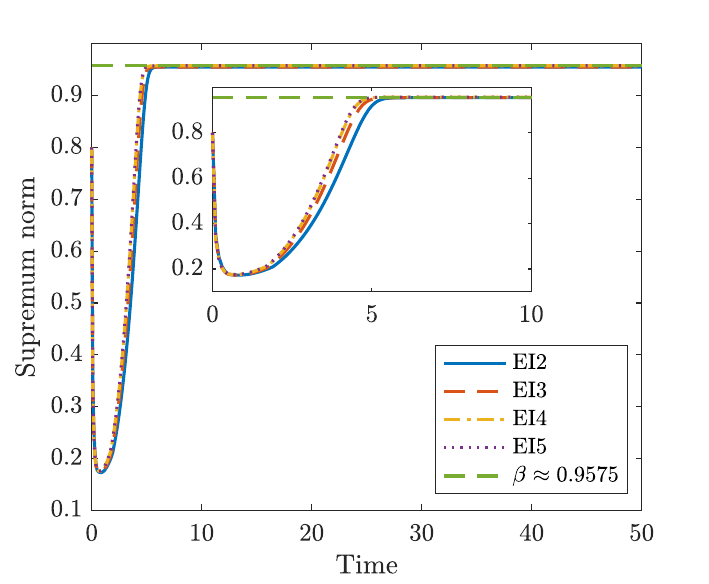}
    \end{subfigure}
    \caption{Evolution of the energy and supremum norm of numerical solutions for \eqref{eq:AC_log}, computed using the EI2 to EI5 schemes with left Gauss–Radau quadrature (time step size $\tau=0.1$). }
    \label{fig:log_energy_MP_01}
\end{figure}

\begin{figure}[htbp]
    \centering
    \begin{subfigure}[b]{0.48\textwidth}
        \centering
        \includegraphics[width=\textwidth]{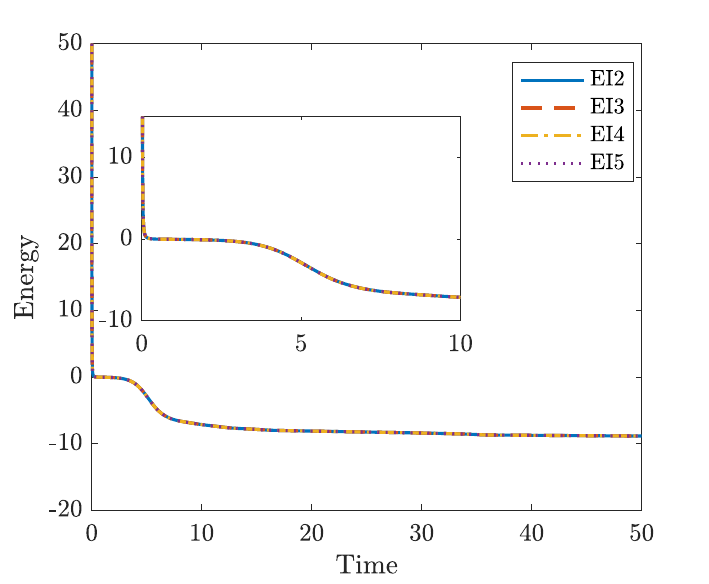}
    \end{subfigure}
    \hfill
    \begin{subfigure}[b]{0.48\textwidth}
        \centering
        \includegraphics[width=\textwidth]{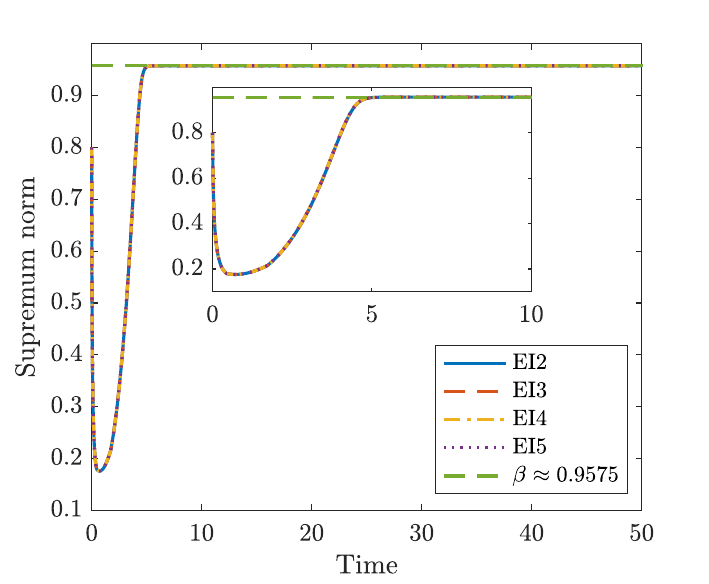}
    \end{subfigure}
    \caption{Evolution of the energy and supremum norm of numerical solutions for \eqref{eq:AC_log}, computed using the EI2 to EI5 schemes with left Gauss–Radau quadrature (time step size $\tau=0.01$). }
    \label{fig:log_energy_MP_001}
\end{figure}

\begin{figure}[htbp]
    \centering
    \includegraphics[width=\textwidth]{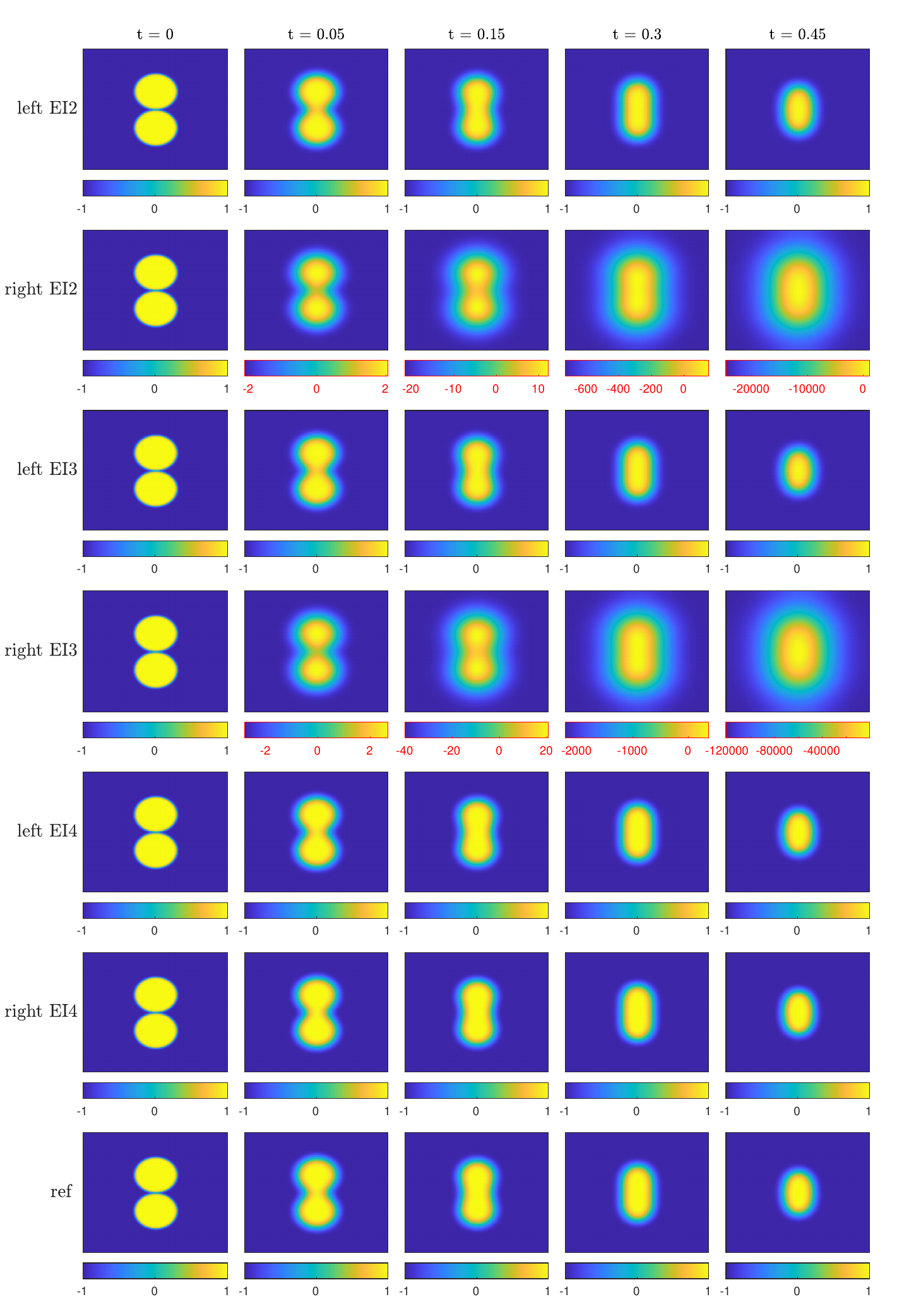}
\caption{Solution snapshots of EI2 to EI4 using left/right Gauss–Radau rules with $\tau = 10^{-3}$, and EI5 (reference) using the left Gauss–Radau rule with $\tau = 10^{-4}$ for solving \eqref{eq:AC_log_2}. (The solutions of right EI2 and right EI3 blow up.)}
    \label{fig:dynamics}
\end{figure}

\begin{figure}[htbp]
    \centering
    \begin{subfigure}[b]{0.48\textwidth}
        \centering
    \includegraphics[width=\textwidth]{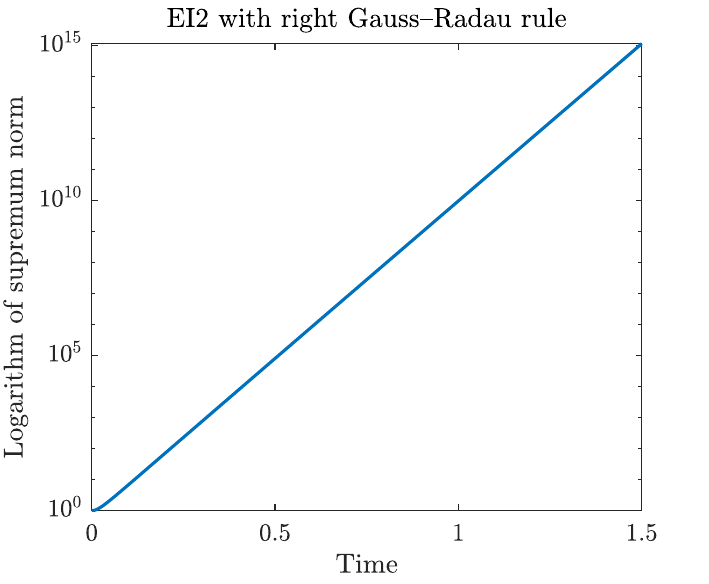}
    \end{subfigure}
    \hfill
    \begin{subfigure}[b]{0.48\textwidth}
        \centering
    \includegraphics[width=\textwidth]{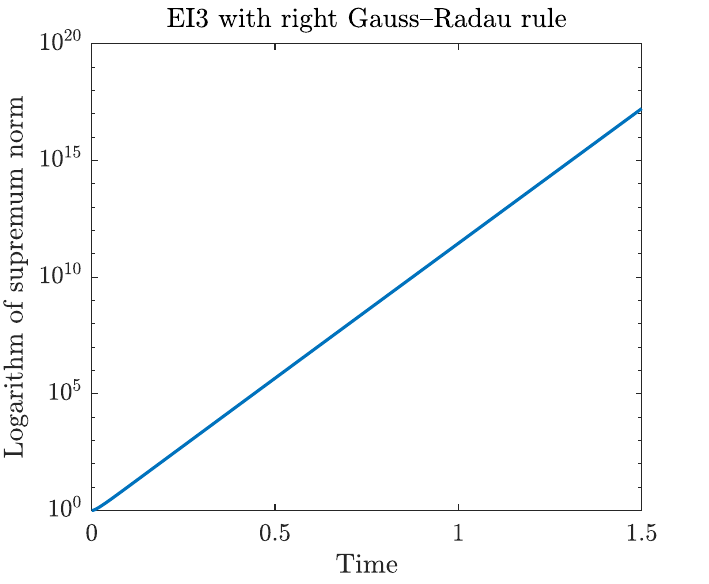}
    \end{subfigure}

    \caption{Evolution of the supremum norm of numerical solutions computed by the EI2 and EI3 using right Gauss--Radau rule with $\tau = 10^{-3}$ for solving \eqref{eq:AC_log_2}.}
    \label{fig:blowup}
\end{figure}

\begin{figure}[htbp]
    \centering
    \begin{subfigure}[b]{0.48\textwidth}
        \centering
        \includegraphics[width=\textwidth]{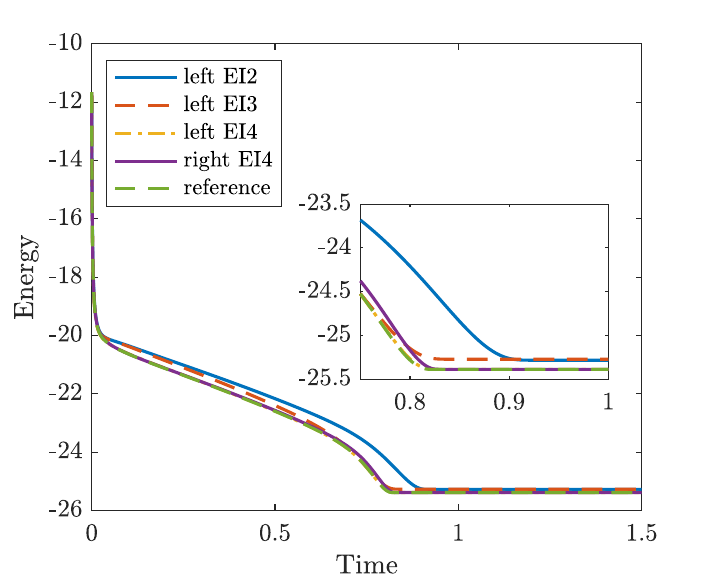}
    \end{subfigure}
    \hfill
    \begin{subfigure}[b]{0.48\textwidth}
        \centering
        \includegraphics[width=\textwidth]{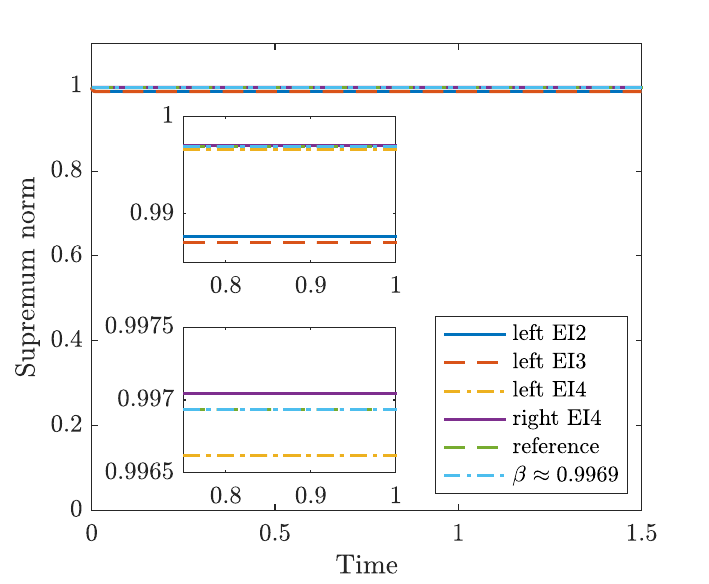}
    \end{subfigure}
\caption{Evolution of the energy and supremum norm of numerical solutions  the computed by EI2 to EI4 schemes using the left Gauss--Radau rule ($\tau = 10^{-3}$), EI4 using the right Gauss--Radau rule ($\tau = 10^{-3}$), and EI5 (reference) using the left Gauss--Radau rule ($\tau = 10^{-4}$) for solving \eqref{eq:AC_log_2}.}
    \label{fig:lGR}
\end{figure}


\section{Concluding remarks}
In this work, we develop and analyze a family of arbitrarily high-order, maximum-bound-preserving time-stepping schemes for solving the Allen–Cahn equation. These schemes are designed within the iterative framework of exponential integrators and incorporate carefully chosen numerical quadrature rules, specifically the Gauss–Legendre and left Gauss–Radau quadrature rules. We rigorously prove that the proposed methods unconditionally preserve the maximum principle without requiring any additional postprocessing techniques, while maintaining arbitrarily high-order temporal accuracy. Rigorous error analysis is provided. 

To validate the theoretical results and demonstrate the practical performance of the proposed methods, we conduct numerical experiments. The experiments confirm that the schemes successfully preserve the maximum principle while achieving high levels of accuracy. Additionally, the results emphasize the critical role of selecting appropriate quadrature rules. Poor choices of quadrature can violate the maximum principle, distort the solution dynamics, and even result in numerical blow-up.

Our numerical experiments confirm that the schemes effectively preserve the maximum principle while achieving high accuracy. Additionally, the experiments underscore the critical role of selecting appropriate quadrature rules. An unsuitable choice of quadrature can violate the maximum principle, distort the solution dynamics, and even lead to numerical blow-up.

For future work, it would be valuable to further investigate the energy dissipation property of numerical methods for the Allen--Cahn equation. Numerical schemes that preserve both the maximum principle and the energy dissipation law could have broader applicability in practical simulations.

\section*{Acknowledgments}
The work of C. Quan is supported by National Natural Science Foundation of China  (Grant No. 12271241), Guangdong Provincial Key Laboratory of Mathematical Foundations for Artificial Intelligence (2023B1212010001), Guangdong Basic and Applied Basic Research Foundation (Grant No. 2023B1515020030), and Shenzhen Science and Technology Innovation Program (Grant No. JCYJ20230807092402004). The work of Z. Zhou is supported by by National Natural Science Foundation of China (Project 12422117), Hong Kong Research Grants Council (15303122) and an internal grant of Hong Kong Polytechnic University (Project ID: P0038888, Work Programme: ZVX3).

\bibliographystyle{siamplain}

\begin{thebibliography}{10}

\bibitem{MR117917_1960}
{\sc J.~Certaine}, {\em The solution of ordinary differential equations with
  large time constants}, in Mathematical methods for digital computers, Wiley,
  New York-London, 1960, pp.~128--132.

\bibitem{ChengShen:CMAME2022}
{\sc Q.~Cheng and J.~Shen}, {\em A new {L}agrange multiplier approach for
  constructing structure preserving schemes, {I}. {P}ositivity preserving},
  Comput. Methods Appl. Mech. Engrg., 391 (2022), pp.~Paper No. 114585, 25,
  \url{https://doi.org/10.1016/j.cma.2022.114585},
  \url{https://doi.org/10.1016/j.cma.2022.114585}.

\bibitem{ChengShen:SINUM2022}
{\sc Q.~Cheng and J.~Shen}, {\em A new {L}agrange multiplier approach for
  constructing structure preserving schemes, {II}. {B}ound preserving}, SIAM J.
  Numer. Anal., 60 (2022), pp.~970--998,
  \url{https://doi.org/10.1137/21M144877X},
  \url{https://doi.org/10.1137/21M144877X}.

\bibitem{CoxMatthews2002ExponentialTime}
{\sc S.~M. Cox and P.~C. Matthews}, {\em Exponential time differencing for
  stiff systems}, J. Comput. Phys., 176 (2002), pp.~430--455,
  \url{https://doi.org/10.1006/jcph.2002.6995},
  \url{https://doi.org/10.1006/jcph.2002.6995}.

\bibitem{DuEtAl2019MaximumPrinciple}
{\sc Q.~Du, L.~Ju, X.~Li, and Z.~Qiao}, {\em Maximum principle preserving
  exponential time differencing schemes for the nonlocal {A}llen-{C}ahn
  equation}, SIAM J. Numer. Anal., 57 (2019), pp.~875--898,
  \url{https://doi.org/10.1137/18M118236X},
  \url{https://doi.org/10.1137/18M118236X}.

\bibitem{DuEtAl2021MaximumBound}
{\sc Q.~Du, L.~Ju, X.~Li, and Z.~Qiao}, {\em Maximum bound principles for a
  class of semilinear parabolic equations and exponential time-differencing
  schemes}, SIAM Rev., 63 (2021), pp.~317--359,
  \url{https://doi.org/10.1137/19M1243750},
  \url{https://doi.org/10.1137/19M1243750}.

\bibitem{EhleLawson1975GeneralizedRungeKutta}
{\sc B.~L. Ehle and J.~D. Lawson}, {\em Generalized {R}unge-{K}utta processes
  for stiff initial-value problems}, J. Inst. Math. Appl., 16 (1975),
  pp.~11--21.

\bibitem{Fasi2024}
{\sc M.~Fasi, S.~Gaudreault, K.~Lund, and M.~Schweitzer}, {\em Challenges in
  computing matrix functions}.
\newblock arXiv:2401.16132v1, 2024.

\bibitem{GallopoulosSaad1992EfficientSolution}
{\sc E.~Gallopoulos and Y.~Saad}, {\em Efficient solution of parabolic
  equations by {K}rylov approximation methods}, SIAM J. Sci. Statist. Comput.,
  13 (1992), pp.~1236--1264, \url{https://doi.org/10.1137/0913071},
  \url{https://doi.org/10.1137/0913071}.

\bibitem{MR102923_1958}
{\sc J.~Hersch}, {\em Contribution \`a{} la m\'ethode des \'equations aux
  diff\'erences}, Z. Angew. Math. Phys., 9a (1958), pp.~129--180,
  \url{https://doi.org/10.1007/BF01600630},
  \url{https://doi.org/10.1007/BF01600630}.

\bibitem{HochbruckLubich1997KrylovSubspace}
{\sc M.~Hochbruck and C.~Lubich}, {\em On {K}rylov subspace approximations to
  the matrix exponential operator}, SIAM J. Numer. Anal., 34 (1997),
  pp.~1911--1925, \url{https://doi.org/10.1137/S0036142995280572},
  \url{https://doi.org/10.1137/S0036142995280572}.

\bibitem{HochbruckEtAl1998ExponentialIntegrators}
{\sc M.~Hochbruck, C.~Lubich, and H.~Selhofer}, {\em Exponential integrators
  for large systems of differential equations}, SIAM J. Sci. Comput., 19
  (1998), pp.~1552--1574, \url{https://doi.org/10.1137/S1064827595295337},
  \url{https://doi.org/10.1137/S1064827595295337}.

\bibitem{HochOster2010EI}
{\sc M.~Hochbruck and A.~Ostermann}, {\em Exponential integrators}, Acta
  Numer., 19 (2010), pp.~209--286,
  \url{https://doi.org/10.1017/S0962492910000048},
  \url{https://doi.org/10.1017/S0962492910000048}.

\bibitem{Jost2013PartialDifferential}
{\sc J.~Jost}, {\em Partial differential equations}, vol.~214 of Graduate Texts
  in Mathematics, Springer, New York, third~ed., 2013,
  \url{https://doi.org/10.1007/978-1-4614-4809-9},
  \url{https://doi.org/10.1007/978-1-4614-4809-9}.

\bibitem{JuEt2015fast}
{\sc L.~Ju, J.~Zhang, L.~Zhu, and Q.~Du}, {\em Fast explicit integration factor
  methods for semilinear parabolic equations}, J. Sci. Comput., 62 (2015),
  pp.~431--455, \url{https://doi.org/10.1007/s10915-014-9862-9},
  \url{https://doi.org/10.1007/s10915-014-9862-9}.

\bibitem{Lawson1967GeneralizedRungeKutta}
{\sc J.~D. Lawson}, {\em Generalized {R}unge-{K}utta processes for stable
  systems with large {L}ipschitz constants}, SIAM J. Numer. Anal., 4 (1967),
  pp.~372--380, \url{https://doi.org/10.1137/0704033},
  \url{https://doi.org/10.1137/0704033}.

\bibitem{LiEtAl2020ArbitrarilyHighOrder}
{\sc B.~Li, J.~Yang, and Z.~Zhou}, {\em Arbitrarily high-order exponential
  cut-off methods for preserving maximum principle of parabolic equations},
  SIAM J. Sci. Comput., 42 (2020), pp.~A3957--A3978,
  \url{https://doi.org/10.1137/20M1333456},
  \url{https://doi.org/10.1137/20M1333456}.

\bibitem{LiEtAl2021StabilizedIntegrating}
{\sc J.~Li, X.~Li, L.~Ju, and X.~Feng}, {\em Stabilized integrating factor
  {R}unge-{K}utta method and unconditional preservation of maximum bound
  principle}, SIAM J. Sci. Comput., 43 (2021), pp.~A1780--A1802,
  \url{https://doi.org/10.1137/20M1340678},
  \url{https://doi.org/10.1137/20M1340678}.

\bibitem{LiuEtAl2024MaximumBound}
{\sc Y.~Liu, C.~Quan, and D.~Wang}, {\em On the maximum bound principle and
  energy dissipation of exponential time differencing methods for the
  matrix-valued {A}llen--{C}ahn equation}, IMA J. Numer. Anal.,  (2024),
  p.~drae090, \url{https://doi.org/10.1093/imanum/drae090}.

\bibitem{MolerVanLoan2003NineteenDubious}
{\sc C.~Moler and C.~Van~Loan}, {\em Nineteen dubious ways to compute the
  exponential of a matrix, twenty-five years later}, SIAM Rev., 45 (2003),
  pp.~3--49, \url{https://doi.org/10.1137/S00361445024180},
  \url{https://doi.org/10.1137/S00361445024180}.

\bibitem{Pope1963ExponentialMethod}
{\sc D.~A. Pope}, {\em An exponential method of numerical integration of
  ordinary differential equations}, Comm. ACM, 6 (1963), pp.~491--493,
  \url{https://doi.org/10.1145/366707.367592},
  \url{https://doi.org/10.1145/366707.367592}.

\bibitem{QuanEtAlMaximumBound}
{\sc C.~Quan, X.~Wang, P.~Zheng, and Z.~Zhou}, {\em Maximum bound principle and
  original energy dissipation of arbitrarily high-order rescaled exponential
  time differencing {R}unge-{K}utta schemes for {A}llen-{C}ahn equations}.
\newblock arXiv:2404.19188v1, 2024.

\bibitem{StoerBulirsch2002IntroductionNumerical}
{\sc J.~Stoer and R.~Bulirsch}, {\em Introduction to numerical analysis},
  vol.~12 of Texts in Applied Mathematics, Springer-Verlag, New York,
  third~ed., 2002, \url{https://doi.org/10.1007/978-0-387-21738-3},
  \url{https://doi.org/10.1007/978-0-387-21738-3}.
\newblock Translated from the German by R. Bartels, W. Gautschi and C.
  Witzgall.

\bibitem{YangEtAl2022ArbitrarilyHighOrder}
{\sc J.~Yang, Z.~Yuan, and Z.~Zhou}, {\em Arbitrarily high-order maximum bound
  preserving schemes with cut-off postprocessing for {A}llen-{C}ahn equations},
  J. Sci. Comput., 90 (2022), pp.~Paper No. 76, 36,
  \url{https://doi.org/10.1007/s10915-021-01746-y},
  \url{https://doi.org/10.1007/s10915-021-01746-y}.

\end{thebibliography}

\end{document}